\documentclass[12pt,reqno]{amsart}
\usepackage{amsmath , amssymb, bbm }
\usepackage{graphics}

\usepackage[mathscr]{eucal}

\newdimen\unit\newdimen\psep\newcount\nd\newcount\ndx\newbox\dotb\newbox\ptbox
\newdimen\dx\newdimen\dy\newdimen\dxx\newdimen\dyy\newdimen\hgt
\newdimen\xoff\newdimen\yoff
\newcommand\clap[1]{\hbox to 0pt{\hss{#1}\hss}}
\newcommand\vdisk[1]{{\font\dotf=cmr10 scaled #1\dotf.}}
\newcommand\varline[2]{\setbox\dotb\hbox{\vdisk{#1}}\xoff=-.5\wd\dotb
\wd\dotb=0pt\yoff=-.5\ht\dotb\psep=#2\ht\dotb}
\newcommand\varpt[1]{\setbox\ptbox\clap{\vdisk{#1}}\setbox\ptbox
\hbox{\raise-.5\ht\ptbox\box\ptbox}}
\newcommand\cpt{\copy\ptbox}
\newcommand\point[3]{\rlap{\kern#1\unit\raise#2\unit\hbox{#3}}}
\newcommand\setnd[4]{\dx=#3\unit\advance\dx-#1\unit\divide\dx by\psep
\dy=#4\unit\advance\dy-#2\unit\divide\dy by\psep \multiply\dx
by\dx\multiply\dy by\dy\advance\dx\dy\nd=1\advance\dx-1sp
\loop\ifnum\dx>0\advance\dx-\nd sp\advance\nd1\advance\dx-\nd
sp\repeat}
\newcommand\dl[4]{{\setnd{#1}{#2}{#3}{#4}\dline{#1}{#2}{#3}{#4}\nd}}
\newcommand\dline[5]{{\nd=#5\hgt=#2\unit\dx=#3\unit\advance\dx-#1\unit
\divide\dx by\nd\dy=#4\unit\advance\dy-#2\unit\divide\dy by\nd
\advance\hgt\yoff\rlap{\kern#1\unit\kern\xoff\loop\ifnum\nd>1\advance\nd-1
\advance\hgt\dy\kern\dx\raise\hgt\copy\dotb\repeat}}}
\newcommand\ellipse[4]{\qellip{#1}{#2}{#3}{#4}\qellip{#1}{#2}{#3}{-#4}%
\qellip{#1}{#2}{-#3}{#4}\qellip{#1}{#2}{-#3}{-#4}}
\newcommand\qellip[4]{{\setnd{0}{0}{#3}{#4}\dx=\unit\dy=0pt\raise\yoff\rlap{%
\kern#1\unit\kern\xoff\raise#2\unit\hbox{\loop\ifnum\dx>0\rlap{\kern#3\dx
\raise#4\dy\copy\dotb}\hgt=\dx\divide\hgt
by\nd\advance\dy\hgt\hgt=\dy \divide\hgt
by\nd\advance\dx-\hgt\repeat\rlap{\raise#4\dy\copy\dotb}}}}}
\newcommand\bez[6]{{\setnd{#1}{#2}{#3}{#4}\ndx=\nd\setnd{#3}{#4}{#5}{#6}
\ifnum\ndx>\nd\nd=\ndx\fi\dx=#3\unit\advance\dx-#1\unit\dy=#4\unit
\advance\dy-#2\unit\dxx=#5\unit\advance\dxx-#1\unit\dyy=#6\unit\advance
\dyy-#2\unit\advance\dxx-2\dx\advance\dyy-2\dy\divide\dxx
by\nd\divide\dyy by\nd\advance\dx.25\dxx\advance\dy.25\dyy\divide\dx
by\nd\divide\dy by\nd \multiply\nd
by2\dx=100\dx\dy=100\dy\dxx=100\dxx\dyy=100\dyy\divide\dxx by\nd
\divide\dyy by\nd\hgt=#2\unit\raise\yoff\rlap{\kern#1\unit\kern\xoff
\raise\hgt\copy\dotb\loop\ifnum\nd>0\advance\nd-1\advance\hgt0.01\dy
\kern0.01\dx\raise\hgt\copy\dotb\advance\dx\dxx\advance\dy\dyy\repeat}}}
\newcommand\ptu[3]{\point{#1}{#2}{\cpt\raise1ex\clap{$\scriptstyle{#3}$}}}
\newcommand\ptd[3]{\point{#1}{#2}{\cpt\raise-1.8ex\clap{$\scriptstyle{#3}$}}}
\newcommand\ptr[3]{\point{#1}{#2}{\cpt\raise-.4ex\rlap{$\ \scriptstyle{#3}$}}}
\newcommand\ptl[3]{\point{#1}{#2}{\cpt\raise-.4ex\llap{$\scriptstyle{#3}\ $}}}
\newcommand\ptlu[3]{\point{#1}{#2}{\raise.8ex\clap{$\scriptstyle{#3}$}}}
\newcommand\ptld[3]{\point{#1}{#2}{\raise-1.6ex\clap{$\scriptstyle{#3}$}}}
\newcommand\ptlr[3]{\point{#1}{#2}{\raise-.4ex\rlap{$\,\scriptstyle{#3}$}}}
\newcommand\ptll[3]{\point{#1}{#2}{\raise-.4ex\llap{$\scriptstyle{#3}\,$}}}
\newcommand\pt[2]{\point{#1}{#2}{\cpt}}

\newcommand\medline{\varline{800}{.5}}
\newcommand\thnline{\varline{400}{.6}}

\varpt{2500}\thnline\unit=2em

\newtheorem{thm}{Theorem}
\newtheorem*{KK}{The Kruskal-Katona Theorem}
\newtheorem{conj}{Conjecture}
\newtheorem{qu}{Question}

\newtheorem{lemma}[thm]{Lemma}
\newtheorem{prop}[thm]{Proposition}
\newtheorem{cor}[thm]{Corollary}

\newtheorem{obs}[thm]{Observation}
\theoremstyle{definition}
\theoremstyle{definition}\newtheorem*{defn}{Definition}

\newcommand{\ds}{\displaystyle}
\newcommand{\ul}{\underline}

\def\A{\mathcal{A}}
\def\B{\mathcal{B}}

\def\d{\partial}

\def\G{\mathcal{G}}
\def\HH{\mathcal{H}}

\def\L{\mathcal{L}}

\def\P{\mathcal{P}}
\def\Q{\mathcal{Q}}
\def\R{\mathcal{R}}
\def\S{\mathcal{S}}

\def\N{\mathbb{N}}

\def\Z{\mathbb{Z}}

\def\le{\leqslant}
\def\ge{\geqslant}

\def\eps{\varepsilon}

\begin{document}
\title[Shadows of ordered graphs]{Shadows of ordered graphs}

\author{B\'ela Bollob\'as}
\address{Trinity College\\ Cambridge CB2 1TQ\\ England, and Department of Mathematical Sciences\\ The University of Memphis\\ Memphis, TN 38152} \email{B.Bollobas@dpmms.cam.ac.uk}

\author{Graham Brightwell}
\address{Department of Mathematics\\ London School of Economics\\ Houghton Street\\London WC2A 2AE\\ England} \email{G.R.Brightwell@lse.ac.uk}

\author{Robert Morris}
\address{IMPA, Estrada Dona Castorina 110, Jardim Bot\^anico, Rio de Janeiro, RJ, Brasil}
\email{rob@impa.br}
\thanks{The first author was supported during this research by NSF grants DMS-0505550, CNS-0721983 and CCF-0728928, and ARO grant W911NF-06-1-0076, and the third by MCT grant PCI EV-8C, a Research Fellowship from Murray Edwards College, Cambridge, and ERC Advanced grant DMMCA}

\begin{abstract}
Isoperimetric inequalities have been studied since antiquity, and in recent decades they have been studied extensively on discrete objects, such as the hypercube. An important special case of this problem involves bounding the size of the shadow of a set system, and the basic question was solved by Kruskal (in 1963) and Katona (in 1968). In this paper we introduce the concept of the shadow $\d\G$ of a collection $\G$ of ordered graphs, and prove the following, simple-sounding statement: if $n \in \N$ is sufficiently large, $|V(G)| = n$ for each $G \in \G$, and $|\G| < n$, then $|\d \G| \ge |\G|$. As a consequence, we substantially strengthen a result of Balogh, Bollob\'as and Morris on hereditary properties of ordered graphs: we show that if $\P$ is such a property, and $|\P_k| < k$ for some sufficiently large $k \in \N$, then $|\P_n|$ is decreasing for $k \le n < \infty$.
\end{abstract}

\maketitle

\section{Introduction}

Of all the closed curves in the plane (of a fixed length), which encloses the largest area? This famous question is known as the isoperimetric problem, and was solved by Steiner~\cite{Steiner} in the 19th century, although the question (and its answer, a circle) was already known in Ancient Greece (see~\cite{Wieg} for a good account). The basic inequality has since been generalized and extended in many beautiful ways; see for example~\cite{Oss,Tala2}.

We shall be interested in discrete analogues of the isoperimetric problem of the following type. Given a graph $G$, and an integer $m \in \N$, what is the minimum (or maximum) size of the vertex- (or edge-) boundary of a set $A \subset V(G)$ of size $m$? This problem has been especially well-studied on the hypercube, $G = \{0,1\}^n$, for which the basic questions were solved by Harper~\cite{Harp1,Harp2} and Hart~\cite{Hart}. More recently, stability results were proved by Friedgut~\cite{Fri1} and by Friedgut, Kalai and Naor~\cite{FKN}. For related work on influences and sharp thresholds, see~\cite{KKL,BKKKL,Tala1,Fri2}, and for results linking isoperimetric inequalities and percolation, see~\cite{ABS,BT1}. A general introduction to the area can be found in~\cite{Bela}.
 
A particular case of this isoperimetric problem was considered by Kruskal~\cite{Krus} in 1963 and Katona~\cite{Kat} in 1968. Given a collection $\A \subset [n]^{(k)}$ of subsets of $\{1,\ldots,n\}$ of size $k$, the shadow $\d \A$ of $\A$ is defined to be
$$\d \A \; := \; \left\{ B \in [n]^{(k-1)} \,:\, B \subset A \textup{ for some } A \in \A \right\}.$$
Kruskal and Katona proved the following well-known theorem. 

\begin{KK}
Let $k \le m \le n$, and let $\A \subset [n]^{(k)}$ be a collection of $k$-sets in $[n]$. If $|\A| \ge {m \choose k}$ then
$$|\d \A| \; \ge \; {m \choose {k-1}}.$$
\end{KK}

Equality holds in the Kruskal-Katona Theorem if (but not only if) $\A$ is an initial segment of the co-lexicographical order. A stability theorem for this result was recently proved by Keevash~\cite{Keev} (see also~\cite{ODW}).

In this paper we shall introduce and study the shadow $\d\G$ of a family of ordered graphs $\G$.  This is one of several possible `two-dimensional' versions of the problem considered by Kruskal and Katona (another is described in Section~\ref{qsec}; see also~\cite{BezBl,FFK,Froh,HHMTZ}). An ordered graph $G$ is a graph together with a linear order on its vertex set; if $|V(G)| = n$, then we identify $V(G)$ with $[n]$, the set $\{1,\ldots,n\}$ with the usual order. Given an ordered graph $G$, and a set of vertices $U \subset V(G)$, let $G[U]$ denote the ordered graph induced by the set $U$, with the inherited order. We shall write $G - v = G[V(G) \setminus \{v\}]$. 

The \emph{shadow} of an ordered graph $G$ is defined to be
$$\d G \; := \; \big\{ H \, : \, H = G - v \textup{ for some } v \in V(G) \big\}$$
and if $\G$ is a collection of ordered graphs then the shadow of $\G$ is
$$\d \G \; := \; \bigcup_{G \in \G} \d G.$$
We shall prove the following sharp lower bound on the size of $\d\G$, in the case where $|\G|$ is not too large. 

\begin{thm}\label{shadow}
Let $n \in \N$, with $n \ge 135$, and let $\G$ be a collection of ordered graphs on $[n]$. If $|\G| < n$, then $|\d \G| \ge |\G|$.
\end{thm}

Theorem~\ref{shadow} is sharp, since there exist families $\G$ of ordered graphs on $[n]$ with $|\d\G| < |\G| = n$, and there exist families $\G_\ell$ with $|\d\G_\ell| = |\G_\ell| = \ell$ for every $\ell < n$ (see Section~\ref{notasec}). The condition $n \ge 135$, on the other hand, is an artifact of the proof, and is almost certainly unnecessary. 

Note that although this result seems simple, ordered graphs are complex objects containing a large amount of information. The special case of Theorem~\ref{shadow} in which each ordered graph has \emph{at most one edge} is equivalent to a special case of an isoperimetric inequality on $\N^3$ first proved by Bollob\'as and Leader using compressions (see Corollary~11 of~\cite{BL}). In order to gain an appreciation of the complexity of the problem, the reader is encouraged to prove Theorem~\ref{shadow} for himself in the case in which each ordered graph has at most two edges, or indeed in the case $n = 4$.

We remark that the technique we shall use to prove Theorem~\ref{shadow} is (to our knowledge) entirely new. It is based on a structural result on shadows of sets (see Lemma~\ref{line}), together with several structural results on pairs of ordered graphs (see Lemmas~\ref{3sacrowd} and~\ref{4inarow}, for example). In particular, none of the standard techniques from the study of hereditary properties of ordered graphs (see below) seem strong enough to prove such a result. 

The class of ordered graphs is an extremely rich and important one in its own right, and includes various well-studied structures (such as sets, permutations, zero-one matrices and partitions) as special cases. For instance, the class of partitions of $[n]$ is in 1-1 correspondence with the ordered graphs on $[n]$ in which all components are cliques, and the class of permutations of $[n]$ is in 1-1 correspondence with the ordered graphs on $[n]$ which do not contain either $H_1$ or $H_2$ as an induced ordered subgraph, where $|V(H_1)| = |V(H_2)| = 3$, $E(H_1) = \{13\}$ and $E(H_2) = \{12,23\}$. (To see this, consider the ordered graph $H(\pi)$ with edge set $\{ij : \pi(i) > \pi(j)\}$ for each permutation $\pi$.) In particular, our study of shadows of ordered graphs was motivated by the earlier work of Klazar~\cite{Klaz1,Klaz2}, Marcus and Tardos~\cite{MT}, Balogh and Bollob\'as~\cite{BB}, Pach and Tardos~\cite{PT}, Balogh, Bollob\'as and Morris~\cite{BBMorder, BBMparts}, and others, on \emph{hereditary properties} of ordered graphs, and related structures. 

A hereditary property of ordered graphs is a collection $\P$ of ordered graphs closed under taking induced ordered subgraphs. For example, the collection of all ordered graphs not containing a given ordered graph $H$ as an induced ordered subgraph is hereditary (these properties are called `primitive'). Observe that the bijections described above both commute with the operation `taking an induced sub-structure' (see for example~\cite{BBMparts, Klaz2, MT}), and hence any hereditary property of partitions or permutations can be thought of as a hereditary property of ordered graphs.

We write $\P_n$ for the collection of ordered graphs in $\P$ with vertex set $[n]$, and call the function $n \mapsto |\P_n|$ the \emph{speed} of $\P$. The speed is a natural measure of the `size' of $\P$, and was introduced (in the context of primitive properties of labelled graphs) by Erd\H{o}s, Kleitman and Rothschild~\cite{EKR} in 1976, and studied further in~\cite{EFR,Alek,BT2,BBS,ABBM} (see also~\cite{ICM,BGP,PS}). The study of general hereditary properties of graphs, and of speeds of labelled graphs below $n^n$, was initiated by Scheinerman and Zito~\cite{SZ}; much stronger results were later proved by Balogh, Bollob\'as and Weinreich~\cite{BBW1,BBW2}.

The following result about hereditary properties of ordered graphs follows easily from Theorem~\ref{shadow}.

\begin{cor}\label{hered}
Let $\P$ be a hereditary property of ordered graphs, and let $135 \le k \in \N$. If $|\P_k| < k$, then $|\P_n| \le |\P_k|$ for every $n \ge k$.
\end{cor}

To explain the significance of this result, we first recall the following result of Balogh, Bollob\'as and Morris~\cite{BBMorder}. If $\P$ is a hereditary property of ordered graphs, then either $|\P_n|$ is a polynomial for sufficiently large values of $n$, i.e.,
$$|\P_n| = \sum_{i = 0}^k a_i {n \choose i}$$ 
for some $k \in \N_0$, some integers $a_0, \ldots, a_k$, and all $n \ge n_0(\P)$, or $|\P_n| \ge F_n$ for every $n \in \N$, where $F_n \sim \left( \frac{1}{2}(1 + \sqrt{5}) \right)^n$ is the Fibonacci sequence. Moreover, if $|\P_n|$ is unbounded then $|\P_n| \ge n$ for every $n \in \N$.  

In other words, a restriction on the value of $|\P_k|$ for one specific value $k \in \N$ allows us to deduce an asymptotic upper bound on $|\P_n|$ as $n \to \infty$. It is possible to read out explicit bounds on $|\P_n|$ from the proof in~\cite{BBMorder}; however, these bounds only hold when $n$ is larger than some (again explicit, but very large) function of $k$ and $|\P_k|$. This is a common failing in theorems of this type (see, for example,~\cite{BBW1,MT}).

Corollary~\ref{hered} corrects this shortcoming in the simplest case: when the function $n \mapsto |\P_n|$ is eventually constant. The result of Balogh, Bollob\'as and Morris~\cite{BBMorder} says that if $|\P_k| < k$ for some $k \in \N$, then $|\P_n| = c$ for some $c \in \N$, and every sufficiently large $n \in \N$. Note that this does not rule out the possibility that $c$ is much larger than $k$, or that $|\P_n|$ exhibits large fluctuations before finally settling down at $c$. Corollary~\ref{hered} says that in fact, $|\P_n|$ is decreasing on $[k,\infty)$. We remark that this property was by no means inevitable; for example, hereditary properties of words can fluctuate wildly (see Theorem~9 of~\cite{BB}, and also~\cite{Leck}). Finally, we note that by Proposition~\ref{ext2} (below), Corollary~\ref{hered} is best possible.

The rest of the paper is organised as follows. In Section~\ref{notasec} we describe the extremal examples, and some of the notation we shall use throughout the paper. In Section~\ref{setsec} we prove a key lemma on shadows of sets, which may be of independent interest. In Section~\ref{typesec} we describe a partition of the ordered graphs on $[n]$ into `types', and use the result of Section~\ref{setsec} to prove a special case of Theorem~\ref{shadow}. The most substantial part of the paper is Section~\ref{smallsec}, in which we deal with ordered graphs with many small `homogenous blocks'. After all this preparation, the pieces are put  together in Section~\ref{proofsec} to prove Theorem~\ref{shadow} and Corollary~\ref{hered}. Finally, in Section~\ref{qsec}, we discuss possible avenues for further research, including a conjectured extension to general linear speeds.

\section{Extremal examples, and notation}\label{notasec}

Before proceeding with the proofs of Theorem~\ref{shadow} and Corollary~\ref{hered}, we shall first show that they are both sharp, by describing here some of the extremal properties. 

\begin{prop}\label{ext1}
For every $\ell, n \in \N$, with $\ell < n$, there exist collections $\G$ and $\G_\ell$ of ordered graphs on $[n]$ such that 
\begin{itemize}
\item $|\G| = n$ and $|\G_\ell| = \ell$.\\[-2ex]
\item $|\d\G| < |\G|$ and $|\d\G_\ell| = |\G_\ell|$.
\end{itemize}
\end{prop}

\begin{proof}
For each $k \in [n]$, let $G_k$ denote the ordered graph on $[n]$ with edge set $E(G_k) = \{ij : 1 \le i < j \le k\}$, and for each $\ell \le n$, set $\G_\ell = \{ G_k : k \in [\ell] \}$.  Now we see that $|\G_n| = n$ and $|\d \G_n | = n - 1$, whereas $| \d \G_\ell |  = | \G_\ell | = \ell$ for each $\ell < n$.
\end{proof}

To see that Corollary~\ref{hered} is also sharp, consider the property $\P^{(1)} = \bigcup_{n \in \N} \G_n$, where $\G_n$ is as defined in the proof of Proposition~\ref{ext1} above. Thus $\P^{(1)}_n$ is the collection of ordered graphs on $[n]$, with edge set $E^{(1)}_k =  \{ij : 1 \le i < j \le k\}$ for some $1 \le k \le n$. Similarly, for each $2 \le j \le 6$, consider the family $\P^{(j)}$, where $E^{(2)}_k = \{ij : i \le k < j\}$, $E^{(3)}_k = \{1k\}$, $E^{(4)}_k = \{k(k+1)\}$, $E^{(5)}_k = \{1j : 1 < j \le k\}$, and $E^{(6)}_k = \{1j : k < j\}$. (Note that the empty ordered graphs are in $\P^{(j)}$ for each $j$. Two of these six properties are pictured below.) In each case, the property $\P^{(j)}$ is hereditary, and $|\P^{(j)}_n| = n$ for every $n \in \N$.  

 Say that two properties of ordered graphs are equivalent under symmetry if one can be obtained from the other by reversing the vertex order, and/or by switching edges and non-edges. It was proved in~\cite{BBMorder} that, up to equivalence under symmetry, these are the only hereditary properties of ordered graphs $\P$ such that $|\P_n| = n$ for every $n \in \N$. 

\begin{prop}\label{ext2}
There is a hereditary property of ordered graphs $\P$ such that $|\P_n| = n$ for every $n \in \N$. 

Moreover, for any decreasing function $f : [k,\infty) \to \N_0$ with $f(k) < k$, there exists a hereditary property of ordered graphs $\Q(f)$ such that $|\Q(f)_n| = f(n)$ for every $n \ge k$.
\end{prop}

\begin{proof}
For the first part, consider any of the properties $\P^{(1)},\ldots,\P^{(6)}$ described above. To prove the second part, let $\Q(f)_n = \P^{(1)}_n$ if $n < k$, and let $\Q(f)_n$ denote the collection of ordered graphs on $[n]$ with edge set $E^{(1)}_\ell$ for some $1 \le \ell \le f(n)$, otherwise. Since $f$ is decreasing on $[k,\infty)$, it follows that $\Q(f)$ is hereditary, as required.
\end{proof}

\vskip0.2in
\[ \unit = 0.28cm \hskip -13\unit
\medline \ellipse{0}{0}{5}{1} \ellipse{11}{0}{5}{1}
\point{0}{0} {$ \dl{-3.9}{0.4}{-2.9}{-0.6} $}
\point{2.2}{0} {$ \dl{-4}{0.6}{-2.8}{-0.6} $}
\point{4.6}{0} {$ \dl{-4}{0.6}{-2.8}{-0.6} $}
\point{6.8}{0} {$ \dl{-3.9}{0.6}{-2.9}{-0.4} $}
\point{-3}{-4}{\small Edge set $\{ij : 1 \le i < j \le k\}$}
\hskip +28\unit
\medline \ellipse{-0.5}{0}{4.5}{1} \ellipse{12.5}{0}{4.5}{1}
\varpt{3000} \pt{5}{0} \pt{7}{0} \bez{5}{0}{6}{1}{7}{0}
\point{4.8}{-1.5}{\tiny $k$} 
\point{0.5}{-4}{\small Edge set $\{k(k+1)\}$}
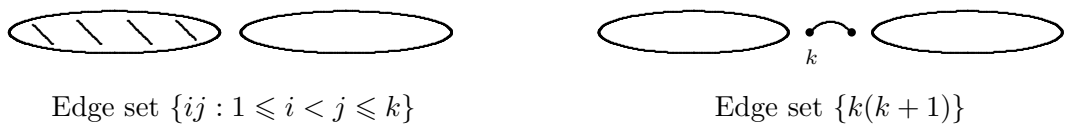
\point{-30}{-7}{\small Figure 1: Two ordered graphs from extremal families, and their edge sets}
\]
\vskip0.3in

We end the section by collecting, for easy reference, some of the notation which we shall use throughout the paper.

As usual, we write $\Z$ for the integers, $\N$ for the natural numbers, and $\N_0 = \N \cup \{0\}$ for the non-negative integers. We write $A^{(k)}$ for the $k$-subsets of a set $A$. Given an ordered graph $G$, we shall use the relation $<$ to denote both the ordering on the vertices of $G$ and the usual order on the integers, and trust that this will cause no confusion. Let $|G| = |V(G)|$ denote the number of vertices of $G$, let $e(G) = |E(G)|$ denote the number of edges of $G$, and let $N(v)$ denote the neighbourhood of a vertex $v \in V(G)$. If $u,v \in V(G)$ then
$$[u,v] \; = \; \{w \in V(G) \,:\, u \le w \le v\},$$ and, similarly, if $u,v \in \Z$, then $[u,v] = \{w \in \Z : u \le w \le v\}$. We shall write $B < C$ if $b < c$ for every $b \in B$ and $c \in C$, for subsets $B,C \subset \Z$ and subsets $B,C \subset V(G)$.

If $x \in \Z^d$ for some $d \in \N$, then $x_i \in \Z$ will denote the $i^{th}$ coordinate of $x$. Let $e_j = (0,\ldots,0,1,0,\ldots,0) \in \Z^d$ denote the vector with a single $1$ in position $j \in [d]$. If $x,y \in \Z^d$, then $x \le y$ if $x_i \le y_i$ for each $i \in [d]$. As usual, $\| x \|_1 = \sum_j |x_j|$ denotes the $L_1$-norm of the vector $x$. 

If $A \subset [n]$, then we shall write $K[A] = A^{(2)}$ for the set of pairs (i.e., potential edges) on $A$. We shall write $K_n$ for $K[[n]] = [n]^{(2)}$, and $\mathbbm{1}[\cdot]$ for the usual indicator function.

Finally, by `symmetry' we mean both the symmetry between edges and non-edges, and that between left and right.

\section{A lemma on shadows of sets}\label{setsec}

In this section we shall prove a lemma on the shadows of sets which, although straightforward, will be an important tool in the proof of Theorem~\ref{shadow}. We begin with some definitions. For each $n,d \in \N_0$, we write
$$\Z^d(n) \; = \; \left\{ x \in \Z^d \, : \, x_i \ge 0\textup{ and }\sum_i x_i = n \right\}.$$
The shadow of a set $A \subset \Z^d(n)$ is the set
$$\d A \; = \; \{ x \in \Z^d(n-1) \, : \, x \le y \textup{ for some } y \in A \}.$$
A \emph{line} in $\Z^d(n)$ is a set
$$\big\{ x \in \Z^d(n) \, : \, x_i = 0\textup{ if }i \notin \{j,k\} \big\}$$
for some $j,k \in [d]$ with $j \neq k$. Observe that if $\L$ is a line in $\Z^d(n)$, then $|\L| = n + 1$ and $|\d \L| = n$. Moreover, if $A$ is a proper subset of a line, then $|\d A| \ge |A|$.

It follows from Corollary~11 of~\cite{BL} that $|\d A| \ge |A| - 1$ for any set $A \subset \Z^d(n)$ with $|A| < 2n$. The following lemma tells us about the case of equality.

\begin{lemma}\label{line}
Let $n,d \in \N$, and let $A \subset \Z^d(n)$ with $|A| < 2n$. Then either $|\d A| \ge |A|$, or $A$ contains a line.
\end{lemma}

\begin{proof}
The proof is by induction on $n + d$. When $n = 1$ or $d = 1$ the result is trivial since $|A| \le 1$, and if $d = 2$ then $A$ is a subset of the line $\Z^2(n)$. If $n = 2$ but $A$ is not a subset of a line, then $|\{i \in [d] : x_i \ge 1\textup{ for some } x \in A\}| \ge 3$, and so $|\d A| \ge 3$.

Fix $n, d \ge 3$. For each $j \in [d]$, we define the `$j^{th}$ face' $F_j$ of $\Z^d(n)$ to be the set $\{x \in \Z^d(n) : x_j = 0\}$, and write $A_j = A \cap F_j$. We shall refer to the set $\{x \in \d A : x + e_j \in A\}$ as the set obtained by `compressing $A$ in direction $j$'.

There are three cases to consider.\\[-1ex]

\noindent \textbf{Case 1}: $A_d$ is empty.\\[-1ex]

\noindent This case is easy: we simply compress in direction $d$. The set obtained has size $|A|$ and is contained in the shadow.\\[-1ex]

\noindent \textbf{Case 2}: $|A_d| = 1$.\\[-1ex]

\noindent Let $u = (a_1, \ldots, a_{d-1},0) \in A$ be the element of $A_d$. Compressing $A$ in direction $d$ gives us a set of size $|A| - 1$. We claim that either $|\d A| \ge |A|$ or
$$S(u) \; := \; \big\{ x \in \Z^d(n) \,: \, x_i \le a_i\textup{ for each } 1 \le i \le d-1 \big\} \; \subset \; A.$$
Indeed, consider an element $b \in S(u) \setminus A$ with $(b_1,\ldots,b_{d-1})$ maximal in the usual partial order on $\Z^{d-1}$. Noting that $b \neq u$, by maximality we have $b + e_j - e_d \in A$ for some $j \in [d-1]$, and so $b - e_d \in \d A$. But $b \notin A$, so $b - e_d$ was not obtained by compressing in direction $d$. Thus $|\d A| \ge |A|$ as claimed.

Now, if $S(u) \subset A$ then $|A| \ge \prod_i (a_i + 1)$. But $\sum_i a_i = n$, and therefore either $|A| \ge 2n$, or $u = (n,0,\ldots,0)$, say. But in the latter case $S(u)$ is a line, so $A$ contains a line, and we are done.\\[-1ex]

\noindent \textbf{Case 3}: $|A_d| \ge 2$.\\[-1ex]

\noindent Assume, without loss of generality, that $|A_d| \ge |A_i|$ for each $i \le d$. We apply the induction hypothesis (for $n$ and $d-1$) to the set $A_d$ and (for $n-1$ and $d$) to the set $(A \setminus A_d) - e_d$. Note that in the latter case we have $|(A \setminus A_d) - e_d| < 2(n-1)$, since $|A_d| \ge 2$.

Suppose first that neither contains a line. Then $|\d A_d| \ge |A_d|$ and $|\d\big((A \setminus A_d) - e_d\big)| \ge |A| - |A_d|$, and the sets $\d A_d$ and $\d\big((A \setminus A_d) - e_d\big) + e_d \subset \d A$ are disjoint, so $|\d A| \ge |A|$. Note also that a line in $A_d$ is a line in $A$, so we are done in this case as well.

Hence we may assume that there is a line in $(A \setminus A_d) - e_d$, i.e., there is a set $\L \subset A \setminus A_d$ such that $\L - e_d$ is a line in $\Z^d(n-1)$. Thus $|\L| = n$, and $\L - e_d$ is contained in $d - 2$ faces of $\Z^d(n-1)$ (since it is a line). Hence $\L$ is contained in either $d-2$ or $d-3$ faces of $\Z^d(n)$ (depending on whether or not the line is constant in direction $d$).

Thus $\L \subset A_j$ for some $j \neq d$, unless $d = 3$ and $\L = (**1)$. In the former case $|A_j \setminus A_d| \ge n$, and so $|A| \ge 2n$, since we chose $|A_d| \ge |A_j|$. In the latter case however, the set $(**0) \cup (**1)$ of size $2n-1$ is in $|\d A|$, and so we are done.
\end{proof}

We conclude this section by remarking that the result above is sharp. To see this, consider the sets $B = (**0) \cup (**1) \setminus (n\,0\,0)$ and $C = (*\,0\,*) \cup (**0) \setminus (n\,0\,0)$, which have size $2n$, have shadows of size $2n - 1$, and do not contain a line. Note also that, for any $\ell < 2n$, there are subsets $B_\ell \subset B$ and $C_\ell \subset C$ such that $|\B_\ell| = |C_\ell| = |\d B_\ell| = |\d C_\ell| = \ell$. Since $\Z^3(n) \subset \Z^d(n)$ these serve as extremal  examples in any dimension $d \ge 3$.

\section{Homogenous blocks and types of ordered graph}\label{typesec}

In this section we shall define the \emph{type} and the \emph{excess} of an ordered graph, notions which will be crucial in what follows. We shall then deduce a special case of Theorem~\ref{shadow} from Lemma~\ref{line}.

We begin by recalling the definition of a homogeneous block in an ordered graph from~\cite{BBMorder} (see also~\cite{BBW1}). Let $G$ be an ordered graph, and for each $x,y \in V(G)$, say that $x \sim y$ if $N(x) \setminus \{y\} = N(y) \setminus \{x\}$. A \emph{homogeneous block} is a maximal collection $B$ of consecutive vertices in $G$ such that $x \sim y$ for every $x,y \in B$. It is easy to see that $\sim$ is an equivalence relation, and that the homogeneous blocks are subsets of equivalence classes, and thus uniquely determined by $G$.

Now, let $G$ be an ordered graph with homogeneous blocks $B_1, \ldots, B_k$, where $B_1 < \cdots < B_k$ (in the order of $G$). Define $H(G)$ to be the ordered graph with loops, with vertex set $[k]$, in which $ij \in E(H)$ if and only if $u_iu_j \in E(G)$ for every $u_i \in B_i$ and $u_i \neq u_j \in B_j$. Furthermore, let $b_G = (b_1,\ldots,b_k) \in \{1,2\}^k$ satisfy $b_i = 1$ if $|B_i| = 1$, and $b_i = 2$ otherwise.

\begin{defn}
The \emph{type} $T(G)$ of $G$ is defined to be the pair $(H(G),b_G)$.
\end{defn}

Let $\HH(G)$ denote the set of homogeneous blocks of an ordered graph $G$.

\begin{defn}
For any ordered graph $G$, let
$$m(G)  \; := \; |V(G)| \,-\, \| b_G \|_1 \; = \; \sum_{B \in \HH(G), |B| \ge 3} \Big( |B| - 2 \Big),$$
and call $m(G)$ the \emph{excess} of $G$.
\end{defn}

Note that if two ordered graphs have the same number of vertices and are of the same type, then they have the same excess. We may therefore write $m_n(T)$ for the excess of an ordered graph on $[n]$ of type $T$.

In order to state the next lemma, we shall also need the concept of shadows and lines within types. First, given a collection of ordered graphs $\G$ on $[n]$, and a type $T$, let
$$\G_T \; = \; \{G \in \G \,:\, T(G) = T\},$$
and let $\phi : \G_T \to \Z^{d_T}\big(m_n(T)\big)$ denote the obvious map which takes $\G_T$ to a subset of $\Z^{d_T}\big(m_n(T)\big)$, where $d_T$ is the number of homogeneous blocks of size at least two in an ordered graph of type $T$. To spell it out, if the homogeneous blocks in $G$ of size at least two are $B_1, \ldots, B_k$, and $B_1 < \cdots < B_k$, then $\phi(G) = (|B_1| - 2, \ldots, |B_k| - 2)$. Note that $\phi$ is injective for every type $T$. A \emph{line} in $\G_T$ is a set of ordered graphs $\L \subset \G_T$ such that $\phi(\L) = \{ \phi(G) : G \in \L\}$ is a line in $\Z^{d_T}(m_n(T))$, as defined in Section~\ref{setsec}.

Finally, we define the operation $\d_\tau$, which we call taking the \emph{shadow within types}, as follows. Given an ordered graph $G$, let
$$\d_\tau G \; := \; \big\{ H \in \d G \,:\, T(H) = T(G) \big\},$$
and for any collection $\G$ of ordered graphs, let
$$\d_\tau \G \; := \; \bigcup_{G \in \G} \d_\tau G.$$
Note that $\d_\tau \G \subset \d\G$, that the sets $\d_\tau \G_T$ are disjoint, and that $\d_\tau \G = \bigcup_T \d_\tau \G_T$. Observe also that $\d(\phi(\G_T)) = \phi(\d_\tau \G_T)$.

We can now deduce the following lemma from Lemma~\ref{line}.

\begin{lemma}\label{Gline}
Let $n \in \N$, and let $\G$ be a collection of ordered graphs on $[n]$. Then either $|\d \G| \ge |\G|$, or there exists a type $T$ such that $|\d_\tau\G_T| < |\G_T|$. 

If $|\d_\tau\G_T| < |\G_T|$, then either $|\G_T| \ge 2m_n(T)$, or $\G_T$ contains a line.
\end{lemma}

\begin{proof}
Since $\d_\tau \G \subset \d\G$, the sets $\G_T$ partition $\G$, and the sets $\d_\tau \G_T$ partition $\d_\tau \G$, it follows immediately that either $|\d\G| \ge |\G|$ or there exists a type $T$ such that $|\d_\tau \G_T| < |\G_T|$. 

For the second part, we apply Lemma~\ref{line} to the set $\phi(\G_T) \subset \Z^{d_T}(m_n(T))$, and observe that $|\d(\phi(\G_T))| = |\phi(\d_\tau\G_T)| = |\d_\tau \G_T| < |\G_T| = |\phi(\G_T)|$. From the lemma, it follows that either $|\G_T| = |\phi(\G_T)| \ge 2m_n(T)$, or $\phi(\G_T)$ contains a line, as required.
\end{proof}

Using Lemma~\ref{Gline}, we shall deduce Theorem~\ref{shadow} when there are not too many homogeneous blocks. First we make the following observation.

\begin{obs}\label{typechange}
Let $G$ be an ordered graph, let $B \subset V(G)$ be a homogeneous block of size at most $2$, and let $v \in B$. Then $T(G - v) \neq T(G)$.
\end{obs}

\begin{proof}
Let $B_1 < \dots < B_k$ be the homogeneous blocks of $G$. Removing a vertex cannot increase the number of homogeneous blocks, since if $x \sim y$ in $G$ (for some $x,y \neq v$) then $x \sim y$ in $G - v$. Thus, if $T(G - v) = T(G)$, then $G - v$ has homogeneous blocks $B_1 \setminus \{v\} < \ldots < B_k \setminus \{v\}$, and moreover $|B_j \setminus \{v\}| = 1$ if and only if $|B_j| = 1$. But if $v \in B_j$ and $|B_j| \le 2$ then this is a contradiction. So $T(G - v) \neq T(G)$, as claimed.
\end{proof}

In order to state the next two lemmas, we shall need to define the following family of ordered graphs. Let $\Q'_n$ denote the ordered graphs $G$ on $[n]$ with one of the following properties:
\begin{itemize}
\item[$(a)$] $|\HH(G)| = 2$ and $b_G = (2,2)$.
\item[$(b)$] $\HH(G) = \big\{ A,B,\{u\} \big\}$, and $E(G) = \{uv : v \in A\}$.
\item[$(c)$] $\HH(G) = \big\{ A,B,\{u\} \big\}$, where $A < u < B$ and $E(G) = \{uv : v \in A \cup B\}$. 
\item[$(d)$] $\HH(G) = \big\{ A,B, \{u\}, \{w\} \big\}$, and either $E(G) = \{uw\}$ or $E(G) = \{wz : z \in A \cup B\}$.
\item[$(e)$] $E(G) = \{w(w+1)\}$ for some $w \in [n-1]$.
\end{itemize}
Let $\Q_n$ denote the ordered graphs on $[n]$ which are equivalent under symmetry to some ordered graph in $\Q_n'$. 

We say that a vertex $u \in V(G)$ \emph{distinguishes} two homogeneous blocks, $A,B \subset V(G)$, if $N(u) \cap (A \cup B) \in \{A,B\}$. The following two lemmas, together with Lemma~\ref{Gline}, prove Theorem~\ref{shadow} in the case that all ordered graphs in $\G$ have excess at least $n/2$.

\begin{lemma}\label{2mT}
Let $n \in \N$, let $\G$ be a collection of ordered graphs on $[n]$, and let $T$ be a type. If $\G_T$ contains a line $\L$ then either $|\d \G_T| \ge 2m_n(T) + 1$, or $\L \subset \Q_n$.
\end{lemma}

\begin{proof}
Let $G$ be an arbitrary ordered graph in the line $\L$, and observe that, since $G$ is in a line, $G$ has at most two homogeneous blocks with three or more elements. We shall refer to these `large' blocks as $A$ and $B$, and we shall let $|A| = a+2$ and $|B| = b+2$, where $a,b \ge 0$, and $a + b = m := m_n(T)$.

Suppose first that there exists a vertex $v \in V(G) \setminus (A \cup B)$ such that $A$ and $B$ do not merge in $G - v$. We claim that $|\d \G_T| \ge 2m+1$. Indeed, $\d \G_T$ contains $m$ ordered graphs of type $T(G)$ (remove a vertex from $A \cup B$), and it also contains $m + 1$ ordered graphs of type $T(G - v)$ (remove $v$ from each ordered graph in the line). By Observation~\ref{typechange}, these are all distinct.

So assume from now on that no such vertex $v$ exists. The rest of the proof is a rather tedious case analysis to show that $G \in \Q_n$. First observe that if $G[A \cup B]$ is neither complete nor empty then $V(G) = A \cup B$, since $A$ and $B$ cannot merge. Thus $|\HH(G)| = 2$ and $b_G = (2,2)$, so $G \in \Q_n$. 

Next suppose that $G[A \cup B]$ is either complete or empty. Thus, either there exists a homogeneous block $C$ with $A < C < B$, or there exists a vertex $u \in V(G)$ which distinguishes $A$ and $B$. In the latter case, note that $A$ and $B$ do not merge in $G - v$ for any vertex $v \in V(G) \setminus (A \cup B \cup \{u\})$, and hence $V(G) = A \cup B \cup \{u\}$. By symmetry we may assume that $E(G) = \{uv : v \in A\}$, and so $G \in \Q_n$ (case $(b)$).

We may therefore assume, using symmetry and without loss of generality, that $G[A \cup B]$ is empty, and that there exists a homogeneous block $C$ with $A < C < B$. We may assume moreover that $N(u) \cap (A \cup B) \in \{\emptyset,A \cup B\}$ for every $u \in V(G)$. 

Suppose first that there exists $v \in V(G) \setminus (A \cup B \cup C)$. Since $A$ and $B$ merge when $v$ is removed, it follows that $G[A \cup B \cup C]$ is empty. But $C$ is a homogeneous block, distinct from $A$ and $B$, so there must be a vertex $w \in V(G) \setminus (A \cup B \cup C)$ which distinguishes $B$ and $C$. If $|C| = 2$, then removing one of the vertices of $C$ does not cause $A$ and $B$ to merge (since $w$ still distinguishes $B$ and $C$), which is a contradiction, so $C = \{u\}$, say. Moreover, $V(G) = A \cup B \cup \{u,w\}$, since any other vertex could be removed without causing $A$ and $B$ to merge. Thus either $E(G) = \{uw\}$ or $E(G) = \{wz : z \in A \cup B\}$, and hence $G \in \Q_n$ (case $(d)$).

So finally, assume that $V(G) = A \cup B \cup C$. Then either $C = \{u\}$ and $N(u) = A \cup B$, (in which case we are in case $(c)$), or $C = \{w,w+1\}$ and $E(G) = \{w(w+1)\}$ (so we are in case $(e)$). In either case $G \in \Q_n$, as required.
\end{proof}

The next lemma deals with the exceptional cases identified by Lemma~\ref{2mT}.

\begin{lemma}\label{Qn}
Let $n \in \N$, let $\G$ be a collection of ordered graphs on $[n]$, and let $T$ be a type. If $\G_T \cap \Q_n$ contains a line, then $|\d\G_T| \ge n - 3$.

If moreover $|\G| < n$ and $m(G) \ge 2$ for every $G \in \G$, then $|\d\G| \ge |\G|$.
\end{lemma}

\begin{proof}
Let $\L$ be a line in $\G_T \cap \Q_n$ and let $G \in \L$, so $T(G) = T$. Suppose first that $|\HH(G)| = 2$ and $b_G = (2,2)$, and note that in this case $\G_T = \L$, since there are no other ordered graphs of type $T$. Using symmetry, there are only two cases to consider: $E(G) = K[A] \cup K[B]$ and $E(G) = K[A]$. In either case we have $|\G_T| = n - 3$, $|\d_\tau \G_T| = n - 4$, and $|\d \G_T| = n - 2$. (To see this in the first case, note that $\d\G_T$ contains the ordered graph with edge set $K[A] \cup K[B]$ for all $1 \le |A| \le n - 2$. The other case is the same.) 

Thus, if $|\G| < n$ then either $|\d\G| \ge |\G|$, or there are two ordered graphs $H_1,H_2 \in \G \setminus \G_T$ such that $\d H_j \subset \d \G_T$ for each $j \in \{1,2\}$. Since $m(G) \ge 1$ for each $G \in \G$, it follows that $|\d_\tau H_j| \ge 1$ for each $j$, and hence the type  of $H_j$ must be the same as one of the ordered graphs in $\d \G_T \setminus \d_\tau\G_T$. But $|\d \G_T \setminus \d_\tau\G_T| = 2$, and each has only one large homogeneous block, so there is only one choice for the ordered graphs $\{H_1,H_2\}$. But then $K_{n-1} \in \d H_j \setminus \d\G_T$ for some $j \in \{1,2\}$, so we have a contradiction. This shows that $|\d\G| \ge |\G|$, as claimed. 

The other cases are similar, so we shall skip over some of the details. Suppose next that $\HH(G) = \big\{ A,B,\{u\} \big\}$, and $E(G) = \{uv : v \in A\}$. Note that in this case also, $\G_T = \L$. It is easy to check that $|\G_T| = n - 4$, that $|\d_\tau\G_T| = n - 5$, and that $|\d\G_T| = n - 2$. So if $|\d\G| < |\G| < n$, then there are three ordered graphs $H_1,H_2,H_3 \in \G \setminus \G_T$ such that $\d H_j \subset \d \G_T$ for each $j \in \{1,2,3\}$. As before, the type  of $H_j$ must be the same as one of the ordered graphs in $\d \G_T \setminus \d_\tau\G_T$. But $|\d \G_T \setminus \d_\tau\G_T| = 3$, and each has only one large homogeneous block, so there is only one choice for the ordered graphs $\{H_1,H_2,H_3\}$. But then the star centred at $u$ (with $n-2$ edges) is in $\d H_j \setminus \d\G_T$ for some $j \in \{1,2,3\}$, so we have a contradiction. This again shows that $|\d\G| \ge |\G|$. 

If $\HH(G) = \big\{ A,B,\{u\} \big\}$, with $A < u < B$ and $E(G) = \{uv : v \in A \cup B\}$, then $\G_T = \L$, and we have $|\G_T| = n - 4$, $|\d_\tau\G_T| = n - 5$, and $|\d\G_T| = n - 2$. The rest of the proof is exactly the same as in the previous case, except the ordered graph in $\d H_j \setminus \d\G_T$ is the star centred at $n$ (or at $1$). 

 If $\HH(G) = \big\{ A,B, \{u\}, \{w\} \big\}$, and either $E(G) = \{uw\}$ or $E(G) = \{wz : z \in A \cup B\}$, then we again have $\G_T = \L$, and hence $|\G_T| = n - 5$, $|\d_\tau\G_T| = n - 6$, and $|\d\G_T| = n - 3$ (if $E(G) = \{uw\}$) or $|\d\G_T| = n - 2$ (if $E(G) = \{wz : z \in A \cup B\}$). The proof that $|\d\G| \ge |\G|$ is now essentially the same as above; the only complication arises if $E(G) = \{uw\}$ and $|\G| = n - 1$, in which case (if $|\d\G| < |\G|$), then there is a set $\B$ of four ordered graphs $H_1,H_2,H_3,H_4 \in \G \setminus \G_T$ such that $\big| \bigcup_j \d H_j \setminus \d \G_T \big| \le 1$. 
 
Let us assume (by symmetry) that $w = 1$, let $\B_1 \subset \B$ denote the ordered graphs in $\B$ which have the same type as some member of $\d \G_T \setminus \d_\tau\G_T$, and let $\B_2 = \B \setminus \B_1$. If $|\B_2| \ge 2$, then $\big| \bigcup_j \d_\tau H_j \setminus \d\G_T \big| \ge 2$, since $m(G) \ge 2$ for every $G \in \G$. So $|\B_1| \ge 3$, which implies that the ordered graphs in $\B_1$ are uniquely determined, and moreover that the ordered graphs with edge set $\{12\}$ and $\{1(n-1)\}$ are both in $\bigcup_j \d H_j \setminus \d \G_T$. Hence we are done as before.
 
Finally, suppose that $E(G) = \{w(w+1)\}$ for some $w \in [n-1]$. The proof is slightly different in this case, since we may have $\G_T \neq \L$. We have $|\G_T| \ge |\L| = n - 5$ and $|\d\G_T| \ge |\d \L| = n - 3$. In order to show that $|\d\G| \ge |\G|$, we consider the set $\B$ of (either three or four) ordered graphs in $\G \setminus \L$.

Let $\B_1 \subset \B$ denote the ordered graphs in $\B$ of type $T$, let $\B_2 \subset \B$ denote those with the same type as a member of $\d \L \setminus \d_\tau \L$, and let $\B_3 = \B \setminus (\B_1 \cup \B_2)$. Note first that if $\B_1 \neq \emptyset$ then $|\d H_j \setminus \d \L | \ge 2$ for any $H_j \in \B_1$ (remove vertex 1 or vertex $n$), so in this case we are done. Hence we may assume that $\G_T = \L$, and that $\B = \B_2 \cup \B_3$. 

Now, if there is an $H_j \in \B_2$ with more than one edge, then $|\d H_j \setminus \d \L | \ge 2$ (remove vertex 1 or vertex $n$). Thus every $H_j \in \B_2$ has only one homogeneous block with more than two vertices, and so $|\B_2| \le 3$. Moreover, if $|\B_2| = 3$ then $|\bigcup_{\B_2} \d H_j \setminus \d \L | = 2$ (they are the ordered graphs with edge sets $\{12\}$ and $\{(n-2)(n-1)\}$), and so $|\B_2| \le 2$. Finally, observe that 
$$\Big| \bigcup_{H \in \B_3} \d_\tau H \setminus \d \L \Big| \; \ge \; \min\big\{ |\B_3| ,2 \big\},$$ 
since $m(G) \ge 2$ for every $G \in \G$. But $|\B_3| = |\B| - |\B_2| \ge |\B| - 2$, so if $|\B_3| < 2$ then
$$|\d\G| \; = \; |\d \L| \,+\, \Big| \bigcup_{H \in \B} \d H \setminus \d \L \Big| \; \ge \; (n - 3) \,+\, |\B_3| \; \ge \; n - 5 + |\B| \; = \;  |\G|,$$ 
as required.
\end{proof}

We can now, as promised, deduce a special case of Theorem~\ref{shadow}. Note that in this case the result is still sharp, and that we do not need to assume that $n$ is sufficiently large.

\begin{cor}
Let $n \in \N$, and let $\G$ be a collection of ordered graphs on $[n]$. Suppose that $|\G| < n$, and that $m(G) \ge n/2$ for every $G \in \G$. Then $|\d \G| \ge |\G|$.
\end{cor}

\begin{proof}
By Lemma~\ref{Gline}, either $|\d \G| \ge |\G|$, or there exists a type $T$ such that $|\d_\tau\G_T| < |\G_T|$, and hence either $|\G_T| \ge 2m_n(T)$ or $\G_T$ contains a line. But $|\G_T| \le |\G| \le n - 1 < 2m_n(T)$ (since $m_n(T) \ge n/2$), so $\G_T$ must contain a line $\L$. 

Now, if $n \le 2$ then the result is trivial, so assume that $n \ge 3$, and hence that $m(G) \ge 2$ for every $G \in \G$. By Lemma~\ref{2mT}, we have $\L \subset \Q_n$ (since $|\G_T| < 2m_n(T)$). Thus $|\d\G| \ge |\G|$ by Lemma~\ref{Qn}, as claimed.
\end{proof}

\section{Ordered graphs with small excess}\label{smallsec}

In this section we shall prove the following pair of propositions, which deal with the case in which $\G$ contains graphs of small excess. The first tells us that several graphs of small excess have a large joint shadow.

\begin{prop}\label{difftypes}
Let $m,n,t \in \N$, and let $\G = \{G_1, \ldots, G_t\}$, where the $G_i$ are distinct ordered graphs on $[n]$. Suppose that $m(G_i) \le m$ for each $i \in [t]$. Then
$$|\d \G| \; \ge \; \frac{t(n-m)^2}{2(n - m) + 32t}.$$
\end{prop}

We remark that, for small values of $t$, there exist families $\G$ with $|\d \G| \le \frac{t(n - m)}{2}$, so the lemma is close to being best possible. The next proposition tells us that if $m(G)$ is small, then the shadow of $G$ also contains a large number of ordered graphs with small excess.

\begin{prop}\label{vsmall}
Let $G$ be an ordered graph on $[n]$, and let $r \in \N$. Then
$$\left| \big\{ H \in \d G : m(H) \le m(G) + 2r + 1 \big\} \right| \; \ge \; \frac{1}{2} \left( 1 - \frac{1}{r} \right) n \: - \: \frac{m(G)}{2}.$$
\end{prop}

The proofs of Propositions~\ref{difftypes} and~\ref{vsmall} are based on two structural results, Lemmas~\ref{3sacrowd} and~\ref{4inarow}, below. We shall first prove Proposition~\ref{vsmall}, which follows without too much difficulty from Lemma~\ref{3sacrowd}; we then prove Proposition~\ref{difftypes}, which will require considerably more effort. 

We begin with the following simple structural result, which will be used in the proof of both of the propositions above.

\begin{lemma}\label{3sacrowd}
Let $G$ be an ordered graph and let $a,b,c \in V(G)$, with $a < b < c$. If $G - a = G - b = G - c$ then $[a,c]$ is a homogeneous block.
\end{lemma}

In order to prove Lemma~\ref{3sacrowd}, we shall need the following, slightly more complicated concept. Define a \emph{semi-homogeneous block} in an ordered graph $G$ to be a collection $B$ of consecutive vertices of $G$ such that, for some set $L \subset \N$ and every $x,y \in B$, $N(x) \setminus B = N(y) \setminus B$, and 
$$xy \in E\big( G[B] \big) \; \Leftrightarrow \; |x - y| \in L.$$
For example, if $G$ is the ordered graph on $[5]$ with edge set $E(G) = \{13,14,24\}$, then $[1,4]$ is a semi-homogeneous block, with $L = \{2,3\}$. 

We shall prove Lemma~\ref{3sacrowd} using the following simpler statement. 

\begin{lemma}\label{gagb}
Let $G$ be an ordered graph and let $a,b \in V(G)$, with $a < b$. If $G - a = G - b$, then $[a,b]$ is a semi-homogeneous block in $G$.
\end{lemma}

\begin{proof}
Let $V(G) = [n]$, and define an equivalence relation $\sim$ on the edges of $K_n$ as follows: $e \sim f$ if $G - a = G - b$ implies that either both or neither of $e$ and $f$ are in $G$. We must show that various collections of edges are in the same equivalence class. \\[-1ex]

\noindent \textbf{Claim 1}: If $e = ij$ and $f = i(j+1)$, with $i < a \le j < b$, then $e \sim f$.\\[-1ex]

\noindent To see this, simply observe that $$ij \in G \; \Leftrightarrow \; ij \in G - b \; \Leftrightarrow \; ij \in G - a \; \Leftrightarrow \; i(j+1) \in G.$$

\noindent \textbf{Claim 2}: If $e = ij$ and $f = (i+1)j$, with $a \le i < b < j$, then $e \sim f$.\\[-1ex]

\noindent This follows from Case 1 by symmetry, or since
$$ij \in G \; \Leftrightarrow \; i(j-1) \in G - b \; \Leftrightarrow \; i(j-1) \in G - a \; \Leftrightarrow \; (i+1)j \in G.$$
It follows from Cases 1 and 2 that $N(x) \setminus [a,b] = N(y) \setminus [a,b]$ for every $x,y \in [a,b]$.\\[-1ex]

\noindent \textbf{Claim 3}: If $e = ij$ and $f = (i+1)(j+1)$, with $a \le i < j < b$, then $e \sim f$.\\[-1ex]

\noindent This again follows easily, since
$$ij \in G \; \Leftrightarrow \; ij \in G - b \; \Leftrightarrow \; ij \in G - a \; \Leftrightarrow \; (i+1)(j+1) \in G.$$
Cases 1, 2 and 3 together imply that $[a,b]$ is a semi-homogeneous block, as required.
\end{proof}

Lemma~\ref{3sacrowd} now follows almost immediately.

\begin{proof}[Proof of Lemma~\ref{3sacrowd}]
By Lemma~\ref{gagb}, $[a,c]$ and $[b,c]$ are semi-homogenous blocks. We shall deduce that $[a,c]$ is a homogeneous block. Indeed, since $[a,c]$ is semi-homogeneous there is a set $L \subseteq [c-a]$ such that, if $x,y \in [a,c]$, then $xy \in E(G)$ if and only if $|x - y| \in L$. Note that $[a,c]$ is homogeneous if and only if $L \in \{\emptyset, [c-a]\}$.

Suppose without loss of generality that $1 \in L$, and let $1 \le x \le c - a$ satisfy $[1,x] \subset L$ but $x + 1 \notin L$. Assume that $x \neq c - a$, and observe that there must exist a vertex $j$, with $a \le j < b \le j + x < c$, such that $j(j+x) \in E(G)$. Since $[b,c]$ is a semi-homogenous block it follows that $j(j+x+1) \in E(G)$. But $j$ and $j+x+1$ are in $[a,c]$, so this contradicts the assumption that $x+1 \notin L$. Hence $x = c - a$, and so $[a,c]$ is a homogeneous block, as claimed.
\end{proof}

Using Lemma~\ref{3sacrowd}, we can now easily prove Proposition~\ref{vsmall}. The proof is via the following simple lemma. Given an ordered graph $G$ and $A \subset V(G)$, define $$\d_{[A]} G \; = \; \big\{ H \in \d G \,:\, H = G - a\textup{ for some } a \in A \big\}.$$

\begin{lemma}\label{a-m/2}
Let $G$ be an ordered graph and $A \subset V(G)$. Then $|\d_{[A]} G| \ge \ds\frac{|A| - m(G)}{2}$.
\end{lemma}

\begin{proof}
Let $A' \subset A$ contain at most two elements of each homogeneous block. By the definition of $m(G)$, we may choose $A'$ so that $|A'| \ge |A| - m(G)$. By Lemma~\ref{3sacrowd}, no three elements of $A'$ give the same ordered graph when they are removed from $G$. Hence $|\d_{[A]} G| \ge |A'|/2$, as claimed.
\end{proof}

Note that this implies that a single ordered graph $G$ has a shadow of size at least $(n - m(G))/2$. Proposition~\ref{vsmall} says that most of these ordered graphs have excess not much larger than $m(G)$.

\begin{proof}[Proof of Proposition~\ref{vsmall}]
First suppose that $v \in V(G)$ lies in a homogeneous block $B$ of size at least two. We claim that $m(G - v) \le m(G) + 3$. Indeed, the only homogeneous blocks which may merge are $B$ and its neighbours, since any other two which were distinguished by $v$ are still distinguished by the elements of $B \setminus \{v\}$. Furthermore, these blocks only merge if $|B| = 2$, and so $|B \setminus \{v\}| = 1$. When three adjacent blocks merge and one of them is a singleton (a homogeneous block of size one), the excess increases by at most three. Hence $m(G - v) \le m(G) + 3$, as claimed.

So suppose that $v \in V(G)$ is a singleton, and suppose that $m(G - v) \ge m(G) + 2r + 2$. Then the removal of $v$ must cause at least $r$ pairs of adjacent blocks to merge (it could also cause the pair either side of it to merge), since when two blocks merge the excess increases by at most two. But each pair of adjacent blocks can be caused to merge by at most one singleton. There are fewer than $n$ adjacent pairs of blocks; thus, there can be at most $n/r$ singletons such that $m(G - v) \ge m(G) + 2r + 2$.

Letting $A = \{v \in V(G) : m(G - v) \le m(G) + 2r + 1\}$, it now follows, by Lemma~\ref{a-m/2}, that
$$|\{H \in \d G : m(H) \le m(G) + 2r + 1\}| \; \ge \; \ds\frac{|A| - m(G)}{2} \; \ge \; \frac{1}{2} \left( 1 - \frac{1}{r} \right) n \: - \: \frac{m(G)}{2},$$ as claimed.
\end{proof}

We now turn to the proof of Proposition~\ref{difftypes}, which is somewhat more complicated. We begin with an easy observation.

\begin{obs}\label{edgein}
Let $G$ and $H$ be ordered graphs on $[n]$, and let $a,b \in [n]$, with $a \le b$. If $G - a = H - b$, then every edge in $G \triangle H$ has an endpoint in $[a,b]$.
\end{obs}

\begin{proof}
Let $i,j \in [n] \setminus [a,b]$, and suppose the edge $ij$ in $G$ corresponds to the edge $f$ in $G - a$. Then (since $i,j \notin [a,b]$) the edge $ij$ in $H$ also corresponds to the edge $f$ in $H - b$. Since $G - a = H - b$, it follows that $ij \notin G \triangle H$.
\end{proof}

Observation~\ref{edgein} has the following simple, but important, consequence.

\begin{lemma}\label{3dj}
Let $G$ and $H$ be ordered graphs on $[n]$, and let $1 \le a_1 \le b_1 < a_2 \le b_2 < a_3 \le b_3 \le n$. If $G - a_i = H - b_i$ for each $i \in \{1,2,3\}$, then $G = H$.
\end{lemma}

\begin{proof}
Suppose $a_1 \le b_1 < a_2 \le b_2 < a_3 \le b_3$. Since $G - a_i = H - b_i$, each edge of $G \triangle H$ has an endpoint in $[a_i,b_i]$, by Observation~\ref{edgein}. But each edge has only two endpoints, so this is impossible, unless $E(G \triangle H) = \emptyset$, in which case $G = H$.
\end{proof}

Lemma~\ref{3dj} deals with the case in which the intervals $(a_i,b_i)$ are disjoint; we need a similar result when we have overlapping intervals. The following lemma provides such a result, but turns out to be somewhat harder to prove.

\begin{lemma}\label{4inarow}
Let $G$ and $H$ be ordered graphs on $[n]$, and let $a_1,\ldots,a_4 \in V(G)$ and $b_1,\ldots, b_4 \in V(H)$ be distinct vertices, with $a_1 < a_2 < a_3 < a_4 \le b_i$ for each $i \in [4]$. Suppose that $G - a_i = H - b_i$ for each $i \in [4]$. Then $[a_2,a_3]$ is a homogeneous block in $G$.
\end{lemma}

\begin{proof}
Say that $e \sim f$ if $G - a_i = H - b_i$ for each $i \in \{1,2,3,4\}$ implies that either both or neither of the edges $e$ and $f$ are in $G$. Clearly $\sim$ is an equivalence relation. We must show that various collections of edges are in the same equivalence class.

Let $(c_1,\ldots,c_4)$ be the reordering of $(b_1,\ldots,b_4)$ so that $c_1 < \cdots < c_4$, and let $\sigma$ be the permutation of $\{1,2,3,4\}$ such that $c_j = b_{\sigma(j)}$ for each $j \in [4]$. We shall write $\hat{\sigma} = \sigma^{-1}$.\\[-1ex]

\noindent \textbf{Claim 1}: Let $e = ij$ and $f = (i+1)j$, with $i \in [a_2,a_3)$ and $j \in [n] \setminus ([a_2,a_3] \cup (c_2,c_3])$. Then $e \sim f$.\\[-1ex]

\noindent The claim follows from the following, slightly more complicated statement.\\[-1ex]

\noindent \ul{Subclaim}: Let $1 \le u < v \le 4$, $i \in [a_u,a_v)$ and $j \in [n] \setminus \big( [a_u,a_v] \cup (\min\{b_u,b_v\},\max\{b_u,b_v\}] \big)$. Then $ij \sim (i+1)j$.

\begin{proof}[Proof of subclaim]
Let
$$\eps_1 \; = \; \mathbbm{1}[j > a_v] \; = \; \mathbbm{1}[a_v < j \le \min\{b_u, b_v\}] \,+\, \mathbbm{1}[j > \max\{b_u,b_v\}]$$ and
$$\eps_2 \; = \; \mathbbm{1}[a_v < j \le b_v] \; = \; \mathbbm{1}[a_v < j \le \min\{b_u, b_v\}] \; = \; \eps_1 \,-\, \mathbbm{1}[j > \max\{b_u,b_v\}].$$ Then
\begin{eqnarray*}
i j \in G & \Leftrightarrow & i(j - \eps_1) \in G - a_v \; \Leftrightarrow \; i(j-\eps_1) \in H - b_v \; \Leftrightarrow \; i(j - \eps_2) \in H, \\
& \Leftrightarrow &  i(j-\eps_1) \in H - b_u \; \Leftrightarrow \; i(j-\eps_1) \in G - a_u \; \Leftrightarrow \; (i+1)j \in G.
\end{eqnarray*}
\end{proof}

Now, let $\S = \{(1,3),(1,4),(2,3),(2,4)\}$, and observe that the intersection of the sets $(\min\{b_u,b_v\},\max\{b_u,b_v\}]$ over all pairs $(u,v) \in \S$ is either $(c_2,c_3]$ (in the case that either $b_1,b_2 \le b_3,b_4$ or $b_3,b_4 \le b_1,b_2$), or empty (otherwise). Hence, if we apply the subclaim to each pair $(u,v) \in \S$ then Claim 1 follows.\\[-1ex]

\noindent \textbf{Claim 2}: If $e = ij$ and $f = (i+1)j$, with $i \in [a_2,a_3)$ and  $j \in (c_2,c_3]$, then $e \sim f$.\\[-1ex]

\noindent First observe that either $b_1,b_2 \le c_2 < c_3 \le b_3,b_4$, or $b_3,b_4 \le c_2 < c_3 \le b_1,b_2$, or we are done by the subclaim. We consider the former two cases separately. \\[-1ex]

\noindent Case 1: $b_1 < b_4$.\\[-1ex]

\noindent By the comments above, we may assume that $\sigma(1),\sigma(2) \in \{1,2\}$ and $\sigma(3),\sigma(4) \in \{3,4\}$. Thus, for any $p \in (a_1,a_{\sigma(4)}]$ and $q \in (c_2,c_4)$, we have
\begin{eqnarray*}
pq \in G & \Leftrightarrow \; & (p-1)(q-1) \in G - a_1 \; \Leftrightarrow \; (p-1)(q-1) \in H - c_{\hat{\sigma}(1)} \; \Leftrightarrow \; (p-1)q \in H\\
& \Leftrightarrow \; & (p-1)q \in H - c_4 \; \Leftrightarrow \; (p-1)q \in G - a_{\sigma(4)} \; \Leftrightarrow \; (p-1)(q+1) \in G.
\end{eqnarray*}
Applying this fact to the edge $ij$, we deduce that $ij \in G \; \Leftrightarrow \; i'j' \in G$, where $i' + j' = i + j$, and either $i' = a_2 - 1$ and $j' \in (c_2,c_3+1]$, or $j' = c_3 + 1$ and $i' \in [a_2,a_3)$. In the same way, we moreover deduce that $(i+1)j \in G \; \Leftrightarrow \; (i'+1)j' \in G$.

We claim that $i'j' \in G \; \Leftrightarrow \; (i'+1)j' \in G$. If $i' = a_2 - 1$ and $j' \in (c_2,c_3+1]$ then this follows by the subclaim with $(u,v) = (1,2)$, since $b_1,b_2 \le c_2$. If $j' = c_3 + 1$ and $i' \in [a_2,a_3)$, then it follows by Claim~1. Hence
$$ij \in G \; \Leftrightarrow \; i'j' \in G \; \Leftrightarrow \; (i'+1)j' \in G \; \Leftrightarrow \; (i+1)j \in G,$$ and so $ij \in G \; \Leftrightarrow \; (i+1)j \in G$, as claimed.\\[-1ex]

\noindent Case 2: $b_4 < b_1$.\\[-1ex]

\noindent This time we may assume that $\sigma(1),\sigma(2) \in \{3,4\}$ and $\sigma(3),\sigma(4) \in \{1,2\}$. Thus, for any $p \in [a_1,a_{\sigma(1)})$ and $q \in (c_1,c_3)$, we have
\begin{eqnarray*}
pq \in G & \Leftrightarrow \; & p(q-1) \in G - a_{\sigma(1)} \; \Leftrightarrow \; p(q-1) \in H - c_1 \; \Leftrightarrow \; pq \in H \\
& \Leftrightarrow \; & pq \in H - c_{\hat{\sigma}(1)} \; \Leftrightarrow \; pq \in G - a_1 \; \Leftrightarrow \; (p+1)(q+1) \in G.
\end{eqnarray*}
Applying this to the edge $ij$, we deduce that $ij \in G \; \Leftrightarrow \; i'j' \in G$, where $j' - i' = j - i$, and either $i' = a_2 - 1$ and $j' \in [c_2,c_3)$, or $j' = c_2$ and $i' \in [a_2,a_3)$. In the same way, we deduce that $(i+1)j \in G \; \Leftrightarrow \; (i'+1)j' \in G$.

We claim that $i'j' \in G \; \Leftrightarrow \; (i'+1)j' \in G$. If $i' = a_2 - 1$ and $j' \in [c_2,c_3)$, this follows by the subclaim with $(u,v) = (1,2)$, since $b_1,b_2 \ge c_3$. If $j' = c_2$ and $i' \in [a_2,a_3)$ then it follows by Claim~1. Hence
$$ij \in G \; \; \Leftrightarrow \; i'j' \in G \; \Leftrightarrow \; (i'+1)j' \in G \; \Leftrightarrow \; (i+1)j \in G,$$ as claimed.\\[-1ex]

\noindent \textbf{Claim 3}: Let $e = ij$, $f = i(j+1)$ and $g = (i+1)(j+1)$, with $a_2 \le i < j < a_3$. Then $e \sim f \sim g$.\\[-1ex]

\noindent Since $G - a_1 = H - b_1$ and $G - a_4 = H - b_4$, for any $a_1 \le p < q < a_4$ we have
\begin{eqnarray*}
pq \in G & \Leftrightarrow \; & pq \in G - a_4 \; \Leftrightarrow \; pq \in H - b_4 \; \Leftrightarrow \; pq \in H \\
& \Leftrightarrow \; & pq \in H - b_1 \; \Leftrightarrow \; pq \in G - a_1 \; \Leftrightarrow \; (p+1)(q+1) \in G.
\end{eqnarray*}
Thus $ij \in G$ if and only if $i'j' \in G$, where $i' = a_2 - 1$, and $j' - i' = j - i$, and also $i(j+1) \in G$ if and only if $i'(j'+1) \in G$. But $j' \in [a_2,a_3)$, so by Claim~1, $i'j' \in G$ if and only if $i'(j'+1) \in G$. Hence
$$ij \in G \; \Leftrightarrow \; i'j' \in G \; \Leftrightarrow \; i'(j'+1) \in G  \; \Leftrightarrow \; i(j+1) \in G$$ as claimed.\\[-1ex]

\noindent Claims 1, 2 and 3 together imply that $[a_2,a_3]$ is a homogeneous block, as required.
\end{proof}

We need one more simple observation.

\begin{lemma}\label{allcliques}
Let $G$ be a graph on $n$ vertices whose components are all cliques. Then $G$ has at least $\ds\frac{n^2}{n + 2e(G)}$ components.
\end{lemma}

\begin{proof}
Let us fix the number of components, $x$, and let $k = n/x$ denote the average size of a component. The expression $n^2 / (n + 2e(G))$ is maximized when the number of edges is minimized, i.e., when the cliques all have (roughly) the same size. Thus, 
$$e(G) \; \ge \; {k \choose 2} \frac{n}{k} \; = \; \frac{n(k-1)}{2}.$$
Rearranging the above expression gives the required result.
\end{proof}

We can now reap our reward: the proof of Proposition~\ref{difftypes}.

\begin{proof}[Proof of Proposition~\ref{difftypes}]
Let $m,n,t \in \N$, and let $\G = \{G_1, \ldots, G_t\}$ be a collection of distinct ordered graphs on $[n]$, with $m(G_i) \le m$ for each $i \in [t]$. We are required to show that
$$|\d \G| \; \ge \; \frac{t(n-m)^2}{2(n - m) + 32t}.$$

First, for each ordered graph $G_i$, choose a set $X_i \subset V(G_i)$, with $|X_i| = m(G_i)$, such that $G_i - X_i$ only has homogeneous blocks of size at most two. This is possible because if $B$ is a homogeneous block in $G$ with $|B| \ge 3$ and $v \in B$, then $B - v$ is a homogeneous block in $G - v$. The set $X_i$ represents the excess of $G_i$.

Now, for each pair $\{i,j\} \subset [t]$, let
$$P(i,j) \; = \; \left\{ (u,v) \in \big( V(G_i) \setminus X_i, V(G_j) \setminus X_j\big) \,:\, G_i - u = G_j - v \right\}.$$
The result follows from the following claim.\\[-1ex]

\noindent \textbf{Claim}: $|P(i,j)| \le 32$ for each $i,j \in [t]$ with $i \neq j$.

\begin{proof}[Proof of claim]
Let $i,j \in [t]$, and let $H$ be the bipartite graph with vertex set $V(G_i) \cup V(G_j)$ and edge set $P(i,j)$. We begin with a simple but crucial observation.\\[-1ex]

\noindent \ul{Subclaim}: $d_H(u) \le 2$ for every $u \in V(H)$.\\[-1ex]

Suppose $(u,v_1), (u,v_2), (u,v_3) \in P(i,j)$. Then $G_j - v_1 =  G_j - v_2 = G_j - v_3$. Therefore, by Lemma~\ref{3sacrowd}, the set $\{v_1,v_2,v_3\}$ is contained in some homogeneous block of $G_j - X_j$. But by the definition of $X_j$, $G_j - X_j$ only has homogeneous blocks of size at most two. Thus $d_H(u) \le 2$ for every $u \in V(H)$, as claimed.\\[-1ex]

Now, suppose $|P(i,j)| \ge 33$. By the subclaim the components of $H$ are paths and cycles, so there exists a matching in $H$ consisting of at least half of the edges of $H$, i.e., with at least 17 edges. Without loss of generality, at least nine of these edges $(a_k,b_k)$ (where $a_k \in G_i$ and $b_k \in G_j$) have $a_k \le b_k$.

Consider the poset formed by these nine intervals $[a_k,b_k]$ in the interval order. A chain of height three in the poset corresponds to three disjoint intervals, and by Lemma~\ref{3dj}, if three such intervals exist then $G_i = G_j$, which is a contradiction. Thus the poset has height at most two, and so it has width at least five.

Let $[a_1,b_1], \ldots, [a_5,b_5]$ be five non-comparable intervals, with $a_1 < \cdots < a_5$ say, such that $G_i - a_k = G_j - b_k$ for each $k \in [5]$. Since the intervals are incomparable, they have a common point, so $a_\ell \le b_k$ for each $k,\ell \in [5]$. Now, by Lemma~\ref{4inarow} the set $\{a_2,a_3,a_4\}$ lies in a homogeneous block. But all homogeneous blocks have size at most two, so this is a contradiction, and so $|P(i,j)| \le 32$, as claimed.
\end{proof}

It is now easy to deduce that $|\d \G| \ge (n-m)t/2 - O(t^2)$. However, we shall work a little to improve the error term. Consider the graph $J$ with vertex set $\bigcup_i V(G_i) \setminus X_i$, and edge set
$$E(J) \: = \: \bigcup_{i \neq j} P(i,j) \, \cup \, \bigcup_i \{uv : u, v \in V(G_i) \setminus X_i, G_i - u = G_i - v\}.$$
Note that the components of $J$ are all cliques, and that $|J| \ge t(n-m)$, by the definition of $X_i$. Moreover, we have $e(J) \le 32{t \choose 2} + |J|/2$, since by the claim there are at most $32{t \choose 2}$ `cross-edges', and by Lemma~\ref{3sacrowd} and the definition of $X_i$, the set $V(G_i) \setminus X_i$ induces a matching for each $i \in [t]$.

Thus, applying Lemma~\ref{allcliques} to the graph $J$, we deduce that $J$ has at least
$$\frac{|J|^2}{2|J| + 32t^2} \; \ge \; \frac{t(n-m)^2}{2(n - m) + 32t}$$
components. Since each component corresponds to a distinct ordered graph in the shadow, this is a lower bound for $|\d\G|$.
\end{proof}

It would be interesting to determine the best possible lower bound in Proposition~\ref{difftypes}. In particular, is it true that when $t = 2$ and $m = 0$, then $|\d \G| \ge n - 1$?

\section{Proof of Theorem~\ref{shadow}}\label{proofsec}

In this section we shall put together the pieces and prove Theorem~\ref{shadow} and Corollary~\ref{hered}. The proof involves some simple calculations, which we collect in the following observation.

\begin{obs}\label{calc}
If $2 \le m \le n/2$ and $n \ge 135$, then
$$\frac{(m+1)(n-m)^2}{2n + 30m + 32} \; \ge \; n - 1.$$
If $t \ge 3$ and $n \ge 4m + 94$ then
$$\frac{t(n-m)^2}{2(n - m) + 32t} \; \ge \; n - 1.$$
\end{obs}

\begin{proof}
For the first part, simple calculus shows that, for fixed $n$, the minimum of the left hand side occurs at one of the extreme points, $m = 2$ and $m = n/2$. When $m = 2$ the inequality reduces to $n^2 - 102n + 104 \ge 0$, and when $m = n/2$ it is implied by the inequality $n^2 - 134n - 120 \ge 0$.

For the second part, by rearranging we see that the required inequality is equivalent to
$$(t-2)n^2 - \big( 2m(t-1) +32t-2 \big) n + \big( tm^2 - 2m + 32t \big) \, \ge \, 0.$$
For $t \ge 3$, this inequality is of the form $an^2 -bn + c \ge 0$, where $a,b,c$ are all positive.  Moreover $b/a$ is decreasing in $t$, so is at most $4m+94$, and therefore the inequality holds whenever $n \ge 4m+94$, as required.
\end{proof}

We can now prove Theorem~\ref{shadow}.

\begin{proof}[Proof of Theorem~\ref{shadow}]
Let $135 \le n \in \N$, let $\G$ be a collection of ordered graphs on $[n]$, and suppose that $|\G| \le n - 1$. We are required to show that $|\d\G| \ge |\G|$.

By Lemma~\ref{Gline}, either $|\d \G| \ge |\G|$, or there exists a type $T$ such that $|\d_\tau\G_T| < |\G_T|$, and in the latter case either $|\G_T| \ge 2m_n(T)$ or $\G_T$ contains a line (for every such type $T$). Among types such that $|\d_\tau\G_T| < |\G_T|$, choose $T$ with $m_n(T)$ maximal.

Suppose first that $m_n(T) \ge n/2$, in which case $|\G_T| \le |\G| \le n - 1 < 2m_n(T)$. Then $\G_T$ must contain a line $\L$, and by Lemma~\ref{2mT}, we have either $|\d \G_T| \ge 2m_n(T) +1 > |\G|$, or $\L \subset \Q_n$. Note that $m(G) \ge n - 6$ for every $G \in \Q_n$. Moreover, if $\L \subset \Q_n$ then, by Lemma~\ref{Qn}, either $|\d\G| \ge |\G|$ or there exists some graph $H \in \G$ with $m(H) \le 1$. 

But $|\d_\tau \G_T| \ge |\d_\tau L| = m_n(T) \ge n - 6$, and by Proposition~\ref{vsmall}, $\d H$ contains at least $(n-2)/4$ ordered graphs with excess at most $6 < n - 7$. Thus 
$$|\d \G| \; \ge \; |\d_\tau \G_T| \,+\, \frac{n-2}{4} \; \ge \; \frac{5n-26}{4} \; > \; n - 1,$$
since $n > 22$, and so we are done in this case.

Hence we may assume from now on that $m := m_n(T) < n/2$. Recall that a line in $\G_T$ contains $m_n(T) + 1$ ordered graphs, so, in either case (unless $m = 0$), $\G_T$ contains at least this many distinct ordered graphs. By Proposition~\ref{difftypes}, it follows that
$$|\d \G| \; \ge \; \frac{(m+1)(n-m)^2}{2n + 30m + 32}.$$
Thus, if $2 \le m \le n/2$ and $n \ge 135$, then $|\d\G| \ge n - 1 \ge |\G|$, by Observation~\ref{calc}.

Therefore we may assume that $m_n(T) \le 1$, and hence that $|\d_\tau \G_{T'}| \ge |\G_{T'}|$ for every type $T'$ with $m_n(T') \ge 2$. But, by Proposition~\ref{vsmall}, applied with $r = 2$, if there exists $G \in \G$ with $m(G) \le 1$, then $\d G$ contains at least $(n-2)/4$ ordered graphs with excess at most 6. If $\G$ contains $t \ge 4$ ordered graphs with excess at most 7, then
$$|\d \G| \; \ge \; \frac{t(n-7)^2}{2(n-7) + 32t} \; \ge \; n - 1,$$
since $n \ge 122$, by Proposition~\ref{difftypes} and Observation~\ref{calc}. Thus, we may assume that all but at most three ordered graphs in $\G$ have excess at least 8, so
$$|\d \G| \; \ge \; \left( \sum_{T \,:\, m_n(T) \ge 8} |\d_\tau \G_T| \right) \,+\, \frac{n-2}{4} \; \ge \; |\G| \,-\, 3 \,+\, \frac{n-2}{4} \; > \; |\G|,$$
and we are done.
\end{proof}

Finally, we deduce Corollary~\ref{hered} from Theorem~\ref{shadow}.

\begin{proof}[Proof of Corollary~\ref{hered}]
Let $\P$ be a hereditary property of ordered graphs, let $135 \le k \in \N$, and suppose that $|\P_k| < k$. We are required to prove that $|\P_{k+1}| \le |\P_k|$.

Indeed, suppose that $|\P_{k+1}| > |\P_k|$. Then $\P_{k+1}$ contains a collection of ordered graphs $\G_{k+1}$ with $|\G_{k+1}| = |\P_k| + 1 < k + 1$. Applying Theorem~\ref{shadow} to $\G_{k+1}$, we obtain
$$|\P_k| \; \ge \; |\d\P_{k+1}| \; \ge \; |\d\G_{k+1}| \; \ge \; |\G_{k+1}| \; = \; |\P_k| + 1,$$
which is a contradiction.
\end{proof}

\section{Open problems, and extensions to higher speeds}\label{qsec}

Theorem~\ref{shadow} is only one step towards understanding shadows of collections of ordered graphs, and we expect that corresponding results should hold for larger families / for hereditary properties with higher speeds. The following questions and conjectures make this explicit.

By the results of~\cite{BBMorder}, there exists a function $f : \N \to \Z$ such that the following holds for every $k \in \N$. If $\P$ is a hereditary property of ordered graphs, and
$$\ds\limsup_{n \to \infty} \big( |\P_n| - (k-1)n \big) \; = \; \infty,$$
then $|\P_n| \ge kn - f(k)$ for every sufficiently large $n \in \N$. Let $f(k)$ be chosen to be minimal, subject to this condition. We remark that $f(1) = 0$, and that $f(k) \ge (k-1)(k+4)/2$. To see the latter, consider the smallest hereditary property containing all the ordered graphs with edge set $E(G) = \{i(i+1),j(j+1)\}$ where $i \le k - 1$ and $i+1 < j$. It is likely, in fact, that this bound is optimal.


\begin{conj}\label{generalk}
Let $n,k \in \N$, and let $f(k) \in \Z$ be as described above. Let $\G$ be a collection of ordered graphs on $[n]$, and suppose that $|\G| < kn - f(k)$. Then $|\d\G| \ge |\G| - k + 1$.
\end{conj}

Note that Theorem~\ref{shadow} is exactly Conjecture~\ref{generalk} in the case $k = 1$ and $n \ge 135$, and so Conjecture~\ref{generalk} includes the extension of Theorem~\ref{shadow} to all $n \in \N$. It is conceivable that the techniques of this paper could be extended in order to prove the conjecture for all $k \in \N$ (and sufficiently large $n$), although one would require a more general version of Lemma~\ref{line}. The following problem, on the other hand, is likely to require further new ideas.

For each $k \in \N$, let $h(k)$ denote the smallest possible quadratic speed of a hereditary property of ordered graphs, $\P$, i.e., the largest integer such that $|\P_n| = \Theta(n^2)$ implies $|\P_k| \ge h(k)$.

\begin{qu}\label{quadqu}
Let $3 \le n \in \N$, and let $h(n)$ be as described above. Let $\G$ be a collection of ordered graphs on $[n]$, and suppose $|\G| < h(n)$. Is it true that $|\d \G| \ge |\G| - n + 3$?
\end{qu}

Question~\ref{quadqu} is partly motivated by the following conjecture about the possible speeds of a hereditary property of ordered graphs, $\P$. Recall that, by the main theorem of~\cite{BBMorder}, if $|\P_m| < F_m$ (the $m^{th}$ Fibonacci number) for some $m \in \N$, then $|\P_n| = \Theta(n^k)$ for some $k \in \N_0$. Suppose $|\P_n| = \Theta(n^k)$; what is the minimal possible value of $|\P_n|$?

Consider the collection $\R(k)$ of ordered graphs with edge set $\{i_1(i_1+1), \ldots, i_t(i_t+1)\}$, where $i_j + 1 < i_{j+1}$ for each $j \in [t-1]$, and $t \le k$. Then $\R(k)$ is hereditary, and has speed ${{n-k} \choose k} + {{n-k+1} \choose {k-1}} + \ldots + {n \choose 0}$. We conjecture that this is the minimal speed of order $n^k$.

\begin{conj}\label{quadconj}
Let $\P$ be a hereditary property of ordered graphs, and let $k \in \N$. If $|\P_n| = \Theta(n^k)$, then
$$|\P_n| \; \ge \; \sum_{i=0}^k {{n-i} \choose i}$$
for every $n \in \N$.
\end{conj}

Conjecture~\ref{quadconj} holds in the case $k = 1$ (by the results of~\cite{BBMorder}). Moreover, if Conjecture~\ref{quadconj} is true, then the bound in Question~\ref{quadqu} (if true) is best possible. To see this, consider the collection $\R(2)$ minus the empty ordered graph: it has ${{n-2} \choose 2} + n - 1$ elements, and there are ${{n-3} \choose 2} + n - 1$ ordered graphs in its shadow. We suspect that Question~\ref{quadqu} is much harder than Conjecture~\ref{quadconj}.

Finally, we remark that the problem considered in this paper is not the only natural two-dimensional analogue of that considered by Kruskal and Katona. For example, one could alternatively consider a collection of (labelled) subgraphs of the complete graph on $\{1,\ldots,n\}$ (all with $k$ vertices, say). 

To be precise, given a graph $G$ with (labelled) vertex set $V(G) \subset \{1,\ldots,n\}$, define 
$$\d G \; := \; \big\{ H \, : \, H = G - v \textup{ for some } v \in V(G) \big\},$$
where $G - v$ is the induced subgraph of $G$ with vertex set $V(G) \setminus \{v\}$. If $\G$ is a collection of such graphs, then define the shadow of $\G$ to be $\d \G \; := \; \bigcup_{G \in \G} \d G$. 

\begin{qu}\label{labelledq1}
Let $n,k \in \N$, and let $\G$ be a family of labelled graphs, with $|V(G)| = k$  and $V(G) \subset [n]$ for each $G \in \G$. Given $|\G|$, what is the minimum possible size of $|\d\G|$?
\end{qu}

In the following question, one gets a tight family by taking $\G$ to be the collection of all graphs with $k$ vertices and labels from $[m]$.

\begin{qu}\label{labelledq2}
Let $\G$ be a family of labelled graphs as described in Question~\ref{labelledq1}, and let $m \in \N$ with $k \le m \le n$. Is it true that if $|\G| \ge \ds 2^{k \choose 2} {m \choose k}$, then $|\d\G| \ge \ds 2^{{k-1} \choose 2}{m \choose {k-1}}$?
\end{qu}

Similar questions could also be asked for families of unlabelled graphs.

\section{Acknowledgements}

The authors would like to thank the anonymous referee for finding an error in the proof of an earlier version of Lemma~\ref{Qn}.

\end{document}